\documentclass[preprint,12pt]{elsarticle}

\usepackage[english]{babel}
\usepackage{graphicx,epstopdf,epsfig}
\usepackage{amsfonts,epsfig,fancyhdr,graphics, hyperref,amsmath,amssymb}
\usepackage{amsthm}
\newtheorem{theorem}{Theorem}[section]
\newtheorem{lemma}[theorem]{Lemma}
\newtheorem{corollary}[theorem]{Corollary}
\newtheorem{proposition}[theorem]{Proposition}

\newtheorem{example}[theorem]{Example}

\journal{Linear Algebra and its Applications}

\begin{document}

\begin{frontmatter}



\title{Zero-dilation indices and numerical ranges}


\author{Kennett L. Dela Rosa}

\affiliation{organization={Institute of Mathematics, University of the Philippines Diliman},
            city={Quezon City},
            postcode={1101}, 
            state={NCR},
            country={Philippines}}

\begin{abstract}
The zero-dilation index $d(A) $ of a matrix $A$ is the largest integer $k$ for which $\begin{bmatrix}0_k& *\\ * & *\end{bmatrix}$ is unitarily similar to $A$. In this study, the zero-dilation indices of certain block matrices are considered, namely, the block matrix analogues of companion matrices and upper triangular KMS matrices, respectively shown as \[\mathcal{C}=\begin{bmatrix} 0& \bigoplus_{j=1}^{m-1}A_j \\ B_0& [B_j]_{j=1}^{m-1}\end{bmatrix}\ \textup{and}\ \mathcal{K}=\begin{bmatrix}0& A& A^2&\cdots& A^{m-1}\\ 0 & 0& A& \ddots& \vdots\\ 0& 0 &0 &\ddots& A^2\\ \vdots& \vdots &\vdots & \ddots& A\\ 0& 0 & 0& \cdots &0\end{bmatrix}\]
where $\mathcal{C}$ and $\mathcal{K}$ are $mn$-by-$mn$ and $A_j,B_j,A$ are $n$-by-$n$. Provided $\bigoplus_{j=1}^{m-1}A_j$ is nonsingular, it is proved that $d(\mathcal{C})$ satisfies the following: if $m\geq 3$ is odd (respectively, $m\geq 2$ is even), then $\frac{(m-1)n}{2}\leq d(\mathcal{C})\leq \frac{(m+1)n}{2}$ (respectively, $ d(\mathcal{C})= \frac{mn}{2}$). In the odd $m$ case, examples are given showing that it is possible to get as zero-dilation index each integer value between $\frac{(m-1)n}{2} $ and $\frac{(m+1)n}{2}$. On the other hand, $d(\mathcal{K})$ is proved to be equal to the number of nonnegative eigenvalues of $(\mathcal{K}+\mathcal{K}^*)/2$. Alternative characterizations of $d(\mathcal{K})$ are given. The circularity of the numerical range of $\mathcal{K} $ is also considered.
\end{abstract}



\end{frontmatter}


\section{Introduction}
\label{}
A $k$-by-$k$ matrix $B$ \textit{dilates} to an $m$-by-$m$ matrix $A$ if $A$ is unitarily similar to a matrix of the form $\begin{bmatrix} B&*\\ *&*\end{bmatrix}$. The \textit{zero-dilation index} $d(A)$ of $A$ is the largest integer $k$ for which the $k$-by-$k$ zero matrix $0_k$ dilates to $A$. The zero-dilation index was introduced in \cite{hG14b}, and it naturally arises from the concept of higher rank numerical ranges. For a given $k\in \{1,\ldots,m\}$, the \textit{rank}-$k$ \textit{numerical range} of $A$, denoted $\Lambda_k(A)$, is defined as \[\Lambda_k(A):=\{\lambda\in \mathbb C:\ \lambda I_k\ \textup{dilates to}\ A\},\] where $I_k$ denotes the $k$-by-$k$ identity matrix. Hence, the zero-dilation index of $A$ is the largest $k$ for which $0\in \Lambda_k(A)$. The higher rank numerical range was first defined in the context of quantum error correction \cite{mC06, mC062}. When $k=1$, $\Lambda_1(A)$ is equal to the classical \textit{numerical range} $W(A)$, defined as
\[W(A):=\{x^*Ax:\ x\in \mathbb{C}^m,\ \|x\|=1\}.\]
At least a century ago, the combined proofs of Toeplitz and Hausdorff showed that $W(A)=\Lambda_1(A)$ is convex. In 2008, it was shown that $\Lambda_k(A)$ is convex for all $k$ (see \cite{cL08,hW08}). 

The zero-dilation index can be characterized as 
\[d(A)=\min\{i_{\geq 0}(\textup{Re}(e^{i\theta}A)):\ \theta\in\mathbb R\},\]
where $i_{\geq 0}(H)$ denotes the number of nonnegative eigenvalues of $H$ counting multiplicities (see \cite[Theorem 2.2]{hG14b} or \cite[Theorem 3.1]{cL08}). As a consequence, $d(A)\leq \lceil \frac{m}{2}\rceil$ provided that $\textup{nullity}(\textup{Re}(e^{i\theta }A))\leq 1$ for some $\theta\in \mathbb {R}$ \cite[Corollary 2.5]{hG14b}. In \cite{hG14b}, it was shown that an $m$-by-$m$ matrix $A$ is unitarily similar to $B\oplus 0_{3d(A)-2m}$ where $B\in M_{3(m-d(A))}$ and $d(B)=2(m-d(A))$ provided that $d(A)>\lfloor \frac{2m}{3}\rfloor.$ This was then used to obtain a characterization of $m$-by-$m$ matrices with $d(A)=m-1.$ The zero-dilation index has been computed for the following classes of matrices: normal matrices,  weighted permutation matrices (with zero diagonals), $S_n$-matrices, companion matrices, and KMS matrices \cite{hG14b,hG14a,hG16}. 

The $m$-by-$m$ \textit{companion matrix} of $L(z)=z^m +\sum_{j=0}^{m-1}a_jz^j$ is the matrix
\[\mathcal{C}=\begin{bmatrix} 0& 1& 0& \cdots& 0\\
0& 0& 1& \cdots& 0\\
\vdots& \vdots& \vdots& \ddots& \vdots\\
0& 0& 0& \cdots& 1\\
-a_0& -a_1& -a_2& \cdots& -a_{m-1}
\end{bmatrix}.\]
In \cite{hG16}, estimates (and in some cases exact values) were given for $d(\mathcal{C})$. If $m\geq 2$, then \cite[Theorem 3.2]{hG16} guarantees that $d(\mathcal{C})\leq \lceil \frac{m}{2}\rceil$; if $m$ is odd (respectively, even), then $d(\mathcal{C})=\frac{m-1}{2}$ or $\frac{m+1}{2}$ (respectively, $d(\mathcal{C})=\frac{m}{2}$). In particular for odd $m$, $d(\mathcal{C})=\frac{m+1}{2}$ if and only if $a_0=a_2=a_4=\cdots =a_{m-1}=0$  \cite[Theorem 3.3]{hG16}. On the other hand, let $a\in \mathbb C$ be given, and consider the upper triangular matrix defined by \[\mathcal{K}_2(a):=\begin{bmatrix} 0&a\\0 &0\end{bmatrix}\ \textup{and}\ \mathcal{K}_m(a):=\begin{bmatrix}0& a& a^2&\cdots& a^{m-1}\\ 0& 0& a& \ddots& \vdots\\ 0& 0 &0 &\ddots& a^2\\ \vdots& \vdots& \vdots& \ddots& a\\ 0& 0 &0 & \cdots &0\end{bmatrix}\ \textup{when}\ m\geq 3.\] In \cite{hG13}, the authors refer to $\mathcal{K}_m(a)$ as the \textit{upper triangular Kac-Murdock-Szeg\"o (KMS) matrix corresponding to $a$}. In the literature, there is another Toeplitz matrix named after Kac, Murdock, and Szeg\"o \cite{KMS}, namely, \[\mathcal{K}_m(a)+\mathcal{K}_m(a)^\top+I_m=\begin{bmatrix}1& a& a^2&\cdots& a^{m-1}\\ a & 1& a& \ddots& \vdots\\ a^2& a &1 &\ddots& a^2\\ \vdots& \ddots& \ddots & \ddots& a\\ a^{m-1}&  \cdots& a^2&a &1\end{bmatrix} \] where $0<a<1$. If $a\in \mathbb C$ above, then recent papers  (see \cite{gF18,gF20,gF19,gF21}) studied properties of the complex symmetric Toeplitz matrix $\mathcal{K}_m(a)+\mathcal{K}_m(a)^\top+I_m$. This study considers $\mathcal{K}_m(a)$, the \textit{upper triangular} KMS matrices or simply KMS matrices for brevity. In \cite{hG10}, the author started the study of the numerical range of KMS matrices by showing that the numerical radius of $\mathcal{K}_m(a)$ is related to the roots of a degree $m$ polynomial in $\cos t$. In \cite{hG13}, the authors launched a systematic study of properties of KMS matrices. It is known that $W(\mathcal{K}_m(a))$ is a circular disk if and only if $m=2$ and $a\neq 0 $ \cite[Theorem 2.3]{hG13}. 

In this paper, we study the zero-dilation indices of block matrix analogues of companion matrices and KMS matrices. Let $m,n\in \mathbb N$ where $m\geq 2$. Let $A=\bigoplus_{j=1}^{m-1}A_j\in M_{m-1}(M_n)$ and $B=[B_j]_{j=0}^{m-1}\in M_{1,m}(M_n)$. A matrix of the form \begin{equation}\label{compT}\mathcal{C}_{A,B}:=\begin{bmatrix} 0& \bigoplus_{j=1}^{m-1}A_j \\ B_0& [B_j]_{j=1}^{m-1}\end{bmatrix} \end{equation} is called a \textit{generalized companion matrix}. This family of matrices includes block companion matrices; in particular, the $m$-by-$m$ \textit{block companion matrix} of $L(z)=I_nz^m +\sum_{j=0}^{m-1}A_jz^j\in M_n$ is given by
\[\begin{bmatrix} 0& I_n& 0& \cdots& 0\\
0& 0& I_n& \cdots& 0\\
\vdots& \vdots& \vdots& \ddots& \vdots\\
0& 0& 0& \cdots& I_n\\
-A_0& -A_1& -A_2& \cdots& -A_{m-1}
\end{bmatrix}.\] Now, for a given $A\in M_n$, the $mn$-by-$mn$ \textit{block upper triangular KMS matrix corresponding to $A$} is the $mn$-by-$mn$ matrix defined by \begin{equation}\label{def:kms}\mathcal{K}_2(A):=\begin{bmatrix} 0& A\\ 0& 0\end{bmatrix};\ \mathcal{K}_m(A):=\begin{bmatrix}0& A& A^2&\cdots& A^{m-1}\\ 0 & 0& A& \ddots& \vdots\\ 0& 0 &0 &\ddots& A^2\\ \vdots& \vdots &\vdots & \ddots& A\\ 0& 0 & 0& \cdots &0\end{bmatrix}\ \textup{when}\ m\geq 3.\end{equation}

Schur complement arguments and other matrix theory techniques are used to generalize the bounds in \cite{hG16} and to show that if $A=\bigoplus_{j=1}^{m-1}A_j$ is nonsingular in \eqref{compT}, then $d(\mathcal{C}_{A,B})$ satisfies the following: if $m\geq 3$ is odd (respectively, $m\geq 2$ is even), then $\frac{(m-1)n}{2}\leq d(\mathcal{C}_{A,B})\leq \frac{(m+1)n}{2}$ (respectively, $ d(\mathcal{C}_{A,B})= \frac{mn}{2}$). For the odd $m$ case, examples are given showing that it is possible to get as zero-dilation index each integer value between $\frac{(m-1)n}{2} $ and $\frac{(m+1)n}{2}$. On the other hand, $d(\mathcal{K}_m(A))$ is computed to be equal to the number of nonnegative eigenvalues of $\textup{Re}(\mathcal{K}_m(A))$. Alternative characterizations of $d(\mathcal{K}_m(A))$ are given provided $A$ satisfies certain conditions. Moreover, necessary and sufficient conditions are given in terms of the matrix parameters $A$ and $B$ in order to determine when $\mathcal{K}_m(A)$ and $\mathcal{K}_m(B)$ are similar or unitarily similar. We also prove a block matrix extension of \cite[Theorem 2.3]{hG13}, namely, if $A $ is nonsingular, then $W(\mathcal{K}_m(A))$ is a circular disk centered at $0$ if and only if $m=2$. 




\section{Preliminaries}

Given $m,n\in \mathbb N $ and a nonempty set $S$, $M_{m,n}(S)$ denotes the set of $m$-by-$n$ matrices over $S$. If $m=n$ or $n=1$, write $M_{m}(S):=M_{m,m}(S)$ and $S^m:=M_{m,1}(S)$, respectively. For convenience, let $M_{m,n}:=\mathbb M_{m,n}(\mathbb C)$  and $M_{m}:=M_{m}(\mathbb C)$. The $k$-by-$k$ zero and identity matrix are denoted by $0_k$ and $I_k$, respectively. The subscript is omitted whenever the size is clear from the context. The (upper triangular) $k$-by-$k$ Jordan block corresponding to $\lambda$ is denoted by $J_k(\lambda)$. The diagonal matrix with $\lambda_1,\ldots,\lambda_n$ in its main diagonal is denoted by $\textup{diag}(\lambda_1,\ldots,\lambda_n).$ For matrices $X=[x_{ij}]$ and $Y$, their direct sum is given by $X\oplus Y=\begin{bmatrix} X& 0\\ 0& Y\end{bmatrix}$ while their Kronecker product is given by $X\otimes Y=[x_{ij}Y]$. If $A=\begin{bmatrix} A_{11}& A_{12}\\ A_{21}& A_{22}\end{bmatrix}$ where $A_{11} $ is nonsingular, then the Schur complement of $A_{11}$ in $A$ is $A_{22}-A_{21}A_{11}^{-1}A_{12}$. For a matrix $X$, let $\sigma(X)$ denote its set of eigenvalues and let $\mbox{Re}(X):=(X+X^*)/2$ and $\mbox{Im}(X):=(X-X^*)/2i$ where $X^*$ denotes the conjugate transpose. Let $H$ be a Hermitian matrix, i.e., $H^*=H$. Counting multiplicities, denote by $i_+(H)$ the number of positive eigenvalues of $H$, $i_-(H)$ the number of negative eigenvalues of $H$, and $i_0(H)$ the number of zero eigenvalues of $H$. Define $i_{\geq 0}(H):=i_+(H)+i_0(H)$ and $i_{\leq 0}(H):=i_-(H)+i_0(H)$. A positive (negative) semidefinite matrix, denoted $H\geq 0$ ($H\leq 0$), is a Hermitian matrix $H$ whose eigenvalues are all nonnegative (nonpositive). The unit circle in $\mathbb C$ is denoted by $\mathbb{T}:=\{\omega\in \mathbb C:\ |\omega|=1\}$.

We have the following consequence of the results in \cite{hG14b,cL08}.

\begin{proposition}\label{approach}
Let $A\in M_m$. If there exist $r\in \{0,1,2,\ldots,m\}$ and $\omega \in \mathbb T$ such that $\textup{nullity}(\textup{Re}(\omega A))\leq r$, then $d(A)\leq \frac{m+r}{2}.$
\end{proposition}
\begin{proof}
Assume $i_0(\textup{Re}(\omega A))=\textup{nullity}(\textup{Re} (\omega A))\leq r$ for some $\omega \in \mathbb{T}$. Then \[i_{+}(\textup{Re}(\omega A))+i_{-}(\textup{Re}(\omega A))=m-i_0(\textup{Re}(\omega A))\geq m-r.\] This implies that one of the summands is at 
least $\frac{m-r}{2}.$ If $i_{-}(\textup{Re}(\omega A))\geq \frac{m-r}{2}$, then \[i_{\geq 0}(\textup{Re}(\omega A))=m-i_{-}(\textup{Re}(\omega A))\leq m-\left(\frac{m-r}{2}\right)=\frac{m+r}{2}.\]
If $i_{+}(\textup{Re}(\omega A))\geq \frac{m-r}{2}$, then $i_{-}(\textup{Re}(-\omega A))=i_{+}(\textup{Re}(\omega A))\geq \frac{m-r}{2}.$ Hence,  
\[i_{\geq 0}(\textup{Re}(-\omega A))=m-i_{-}(\textup{Re}(-\omega A))\leq m-\left(\frac{m-r}{2}\right)=\frac{m+r}{2}.\]
In either case, there exists $\zeta \in \mathbb T$ such that $i_{\geq 0}(\textup{Re}(\zeta A))\leq \frac{m+r}{2}.$ By \cite[Theorem 2.2]{hG14b} or \cite[Theorem 3.1]{cL08}, it follows that $d(A)\leq \frac{m+r}{2}.$
\end{proof}


Proposition \ref{approach} generalizes \cite[Corollary 2.5]{hG14b} since $r=1$ implies $d(A)\leq \lfloor \frac{m+1}{2}\rfloor=\lceil\frac{m}{2}\rceil.$
\section{Generalized companion matrices}

In this section, we study properties of generalized companion matrices.

\subsection{Upper bound on $\textup{nullity}(\mbox{Re}(\omega \mathcal{C}_{A,B}))$}

Let $\mathcal{C}\in M_m$ be the usual companion matrix. \cite[Theorem 3.1]{hG16} guarantees the following:
\begin{enumerate}
    \item[(i)] If $m\geq 3$ is odd, then $\textup{nullity}(\textup{Re}(\omega \mathcal{C}))\leq 1$ for all $\omega \in \mathbb T$.
    \item[(ii)] If $m\geq 2$ is even, then 
    \begin{enumerate}
        \item[(a)] $\textup{nullity}(\textup{Re}(\omega \mathcal{C}))\leq 2$ for all $\omega \in \mathbb T$;
    \item[(b)] $\textup{nullity}(\textup{Re}(\omega \mathcal{C}))\leq 1$ for all but at most finitely $m$ values of $\omega\in \mathbb T$. 
\end{enumerate}
\end{enumerate} Cauchy interlacing was used to prove (i) and (ii.a), but we prove their generalizations without interlacing. 

\begin{proposition}\label{prop:1}
Let $m,n\in \mathbb N$ where $m\geq 2$. Let $A=\bigoplus_{j=1}^{m-1}A_j\in M_{m-1}(M_n)$ and $B=[B_j]_{j=0}^{m-1}\in M_{1,m}(M_n)$. For each $\omega\in \mathbb{T}$, define $Y:=[\omega^{m-j} B_j]$. Then $\omega \mathcal{C}_{A,B}$ and $\mathcal{C}_{A,Y}$ are unitarily similar. In particular, $\textup{Re}(\omega \mathcal{C}_{A,B})$ and $\textup{Re}(\mathcal{C}_{A,Y})$ are unitarily similar.
\end{proposition}

\begin{proof}
Let $\omega\in \mathbb T$ and consider $U=\bigoplus_{j=1}^m\omega^jI_n$. Then $\mathcal{C}_{A,Y}U=\omega U\mathcal{C}_{A,B}$. The claims follow since $U$ is unitary.
\end{proof}

To get a good estimate of $d(\mathcal{C}_{A,B})$ by way of Proposition \ref{approach}, it is crucial to have a small upper bound on $\textup{nullity}(\mbox{Re}(\omega \mathcal{C}_{A,B}))$. 

\begin{proposition}\label{prop:2}
Let $m,n\in \mathbb N$ where $m\geq 3$. Let $A=\bigoplus_{j=1}^{m-1}A_j\in M_{m-1}(M_n)$ and $B=[B_j]_{j=0}^{m-1}\in M_{1,m}(M_n)$. 
Suppose $A_1,\ldots,A_{m-2}$ are nonsingular. If $m$ is odd, then $\textup{nullity}(\textup{Re}(\omega \mathcal{C}_{A,B}))\leq n$ for any $\omega \in \mathbb{T}$; if $m$ is even, $\textup{nullity}(\textup{Re}(\omega \mathcal{C}_{A,B}))\leq 2n$ for any $\omega\in \mathbb{T}.$

\end{proposition}

\begin{proof}
Let $\omega\in \mathbb T$. By Proposition \ref{prop:1}, $\textup{Re}(\omega \mathcal{C}_{A,B})$ and $\textup{Re}(\mathcal{C}_{A,Y})$ are unitarily similar for some $Y\in M_{1,m}(M_n)$. If \[H:= \frac{1}{2}\begin{bmatrix}0& A_1& 0&\cdots& 0\\ A_1^*& 0& A_2& \ddots& \vdots\\ 0& \ddots &\ddots &\ddots& 0\\ \vdots& \ddots& A_{m-3}^*& 0& A_{m-2}\\ 0& \cdots & 0& A_{m-2}^*&0\end{bmatrix}\in M_{m-1}(M_n),\] then $H$ appears as a principal submatrix of $\textup{Re}(\mathcal{C}_{A,Y})$. Hence, \[\textup{nullity} (\textup{Re}(\omega \mathcal{C}_{A,B}))=mn-\textup{rank}(\textup{Re}(\omega \mathcal{C}_{A,B}))\leq mn-\textup{rank}(H).\] 
If $m=3$, note that $H=\begin{bmatrix} 0 & A_1\\ A_1^* & 0\end{bmatrix}$ is nonsingular since $A_1$ is nonsingular, and so \[\textup{nullity} (\textup{Re}(\omega \mathcal{C}_{A,B}))\leq 3n-2n=n.\]
Assume $m\geq 4$. For each $j\in\{1,\ldots,m-3\}$, consider
\[G_j:=\begin{bmatrix} 0& A_j& 0\\ A_j^* &0& A_{j+1}\\ 0& A_{j+1}^*& 0\end{bmatrix},\ L_j:=\begin{bmatrix} A_j^{-1}& 0& 0\\ 0 & (A_j^*)^{-1}& 0\\ -A_{j+1}^*A_j^{-1}& 0 & I_n\end{bmatrix},\] and\[U_j:=\begin{bmatrix} I_n& 0& -(A_j^*)^{-1}A_{j+1}\\ 0 & I_n& 0\\ 0& 0 & I_n\end{bmatrix}.\] All inverses used above exist by assumption on $A_1,\ldots,A_{m-2}$. Observe that \[L_jG_jU_j=\begin{bmatrix}0& I_n& 0\\ I_n& 0& 0\\0 & 0& 0\end{bmatrix}.\] Define
\[\widehat{L_j}:=I_{(j-1)n}\oplus L_j\oplus I_{(m-3-j)n}\ \textup{and}\ \widehat{U_j}:=I_{(j-1)n}\oplus U_j \oplus I_{(m-3-j)n}\]
where the identity summand is omitted if its subscript is $0$. If $m\geq 4$ is odd, then 
\[Q\widehat{L}_{m-4}\widehat{L}_{m-6}\cdots \widehat{L}_{3}\widehat{L}_{1}H\widehat{U}_{1}\widehat{U}_{3}\cdots \widehat{U}_{m-6}\widehat{U}_{m-4}=\textup{Re}(J_{m-1}(0))\otimes I_n\] where $Q:=I_{(m-3)n}\oplus A_{m-2}^{-1}\oplus (A_{m-2}^*)^{-1}.$ Since $m$ is odd, \[\textup{nullity} (\textup{Re}(\omega \mathcal{C}_{A,B}))\leq mn-\textup{rank}(H)= mn-(m-1)n=n.\] If $m\geq 4$ is even, then
\[\widehat{L}_{m-3}\widehat{L}_{m-5}\cdots \widehat{L}_{3}\widehat{L}_{1}H\widehat{U}_{1}\widehat{U}_{3}\cdots \widehat{U}_{m-5}\widehat{U}_{m-3}=(\textup{Re}(J_{m-2}(0))\otimes I_n)\oplus 0_n.\] Since $m$ is even, \[\textup{nullity} (\textup{Re}(\omega \mathcal{C}_{A,B}))\leq mn-\textup{rank}(H) =mn-(m-2)n=2n.\] \end{proof}

The upper bound in the even case of Proposition \ref{prop:2} will be improved in the following results.

\begin{lemma}\label{dettri}
Let $m,n\in \mathbb N$ where $m\geq 2$ is even. Let $A=\bigoplus_{j=1}^{m-1}A_j\in M_{m-1}(M_n)$ be nonsingular and $Y=[Y_j]_{j=0}^{m-1}\in M_{1,m}(M_n)$. Let $P(z):=I_nz^m+\sum_{j=1}^{\frac{m}{2}}P_jz^{2j-2}$ where 
\[P_j:=
(-1)^{\frac{m}{2}-j}A_{m-1}^{-1}A_{m-2}^*A_{m-3}^{-1}A_{m-4}^*A_{m-5}^{-1}\cdots A_{2j}^*A_{2j-1}^{-1}Y_{2j-2}^*\] for all $ j\in \left\{1,2,3,\ldots,\frac{m}{2}\right\}$. Then \[\det(\textup{Re}(\mathcal{C}_{A,Y}))=\frac{(-1)^{\frac{mn}{2}}}{2^{mn}}\prod_{\substack{j=1\\ j\ \textup{odd}}}^{m-1}|\det(A_j)|^2|\det(P(1))|^2.\]

\end{lemma}
\begin{proof}
We use induction on $m$. Consider $m=2$. Then \[\textup{Re}( \mathcal{C}_{A,Y})=\frac{1}{2}\begin{bmatrix} 0 & A_1+Y_0^*\\ Y_0+A_1^*& Y_1+Y_1^*\end{bmatrix}=\frac{1}{2}\begin{bmatrix} A_1+Y_0^*& 0\\  Y_1+Y_1^*& Y_0+A_1^*\end{bmatrix}\begin{bmatrix}0 & I_n\\ I_n& 0\end{bmatrix}.\] Hence, \[\begin{array}{rcl}\det(\textup{Re}( \mathcal{C}_{A,Y}))&=& \frac{(-1)^n}{2^{2n}}|\det(A_1+Y_0^*)|^2\\
&=&\frac{(-1)^n}{2^{2n}}|\det(A_1)|^2|\det(I_n+ A_1^{-1}Y_0^*)|^2\\
&=&\frac{(-1)^n}{2^{2n}}|\det(A_1)|^2|\det(P(1))|^2\\
\end{array}\]
proving the $m=2$ case. Assume $m>2$ is even and the determinant formula works for $(m-2)$-by-$(m-2)$ blocks. If we set $\widetilde{A}:=\bigoplus_{j=1}^{m-3}\widetilde{A}_j=\bigoplus_{j=3}^{m-1}A_j$, $\widetilde{Y}:=[Y_j]_{j=2}^{m-1}$, and $V:=\begin{bmatrix} 0& 0& \cdots& Y_0^*\\ A_2& 0& \cdots&  Y_1^*\end{bmatrix}$, then 
the Schur complement of $2\textup{Re}(J_2(0)\otimes A_1)$ in $2\textup{Re}(\mathcal{C}_{A,Y})$ can be computed as 
$2\textup{Re}(\mathcal{C}_{\widetilde{A},\widetilde{Y}})-V^*(2\textup{Re}(J_2(0)\otimes A_1))^{-1}V=2\textup{Re}(\mathcal{C}_{\widetilde{A},W})$ where \[W:=[W_j]_{j=0}^{m-3}=\begin{bmatrix}Y_2-Y_0(A_1^{-1})^*A_2& Y_3& Y_4& \cdots&  Y_{m-2}& Y_{m-1}-Y_0(A_1^{-1})^*Y_1^*\end{bmatrix}.\] By the induction hypothesis, \[\det(2\textup{Re}(\mathcal{C}_{\widetilde{A},W}))=(-1)^{\frac{(m-2)n}{2}}|\det(A_3A_5\cdots A_{m-3}A_{m-1})|^2|\det(\widetilde{P}(1))|^2\] where $\widetilde{P}(z)=I_nz^{m-2}+\sum_{j=1}^{\frac{m-2}{2}}\widetilde{P}_jz^{2j-2}$ and \[\widetilde{P}_j=(-1)^{\frac{m-2}{2}-j}\widetilde{A}_{m-3}^{-1}\widetilde{A}_{m-4}^*\cdots \widetilde{A}_{2j}^*\widetilde{A}_{2j-1}^{-1}W_{2j-2}^*\] for all $j\in \{1,2,3,\ldots, \frac{m-2}{2}\}$. Note that $\widetilde{P}_1=P_1+P_2$ and $\widetilde{P}_j=P_{j+1}$ for all $j\in \{2,3,\ldots,\frac{m-2}{2}\}$. Hence, \[\widetilde{P}(1)=I_n+\sum_{j=1}^{\frac{m-2}{2}}\widetilde{P}_j=I_n+\widetilde{P}_1+\sum_{j=2}^{\frac{m-2}{2}}\widetilde{P}_j=I_n+P_1+P_2+\sum_{k=3}^{\frac{m}{2}}P_k=P(1).\] By the Schur determinant formula \cite[Section 0.8.5]{HJ1} and the induction hypothesis, we obtain 
\[\begin{array}{rcl}\det(\textup{Re}(\mathcal{C}_{A,Y}))&=&\frac{1}{2^{mn}}\det(2\textup{Re}(J_2(0)\otimes A_1))\det(2\textup{Re}(\mathcal{C}_{\widetilde{A},W}))\\
&=&\frac{(-1)^n}{2^{mn}}|\det(A_1)|^2\det(2\textup{Re}(\mathcal{C}_{\widetilde{A},W}))\\
&=&\frac{(-1)^{\frac{mn}{2}}}{2^{mn}}|\det(A_1A_3A_5\cdots A_{m-3}A_{m-1})|^2|\det(\widetilde{P}(1))|^2 \\
&=&\frac{(-1)^{\frac{mn}{2}}}{2^{mn}}|\det(A_1A_3A_5\cdots A_{m-3}A_{m-1})|^2|\det(P(1))|^2.
\end{array}\]
\end{proof}

Let $L(z)=I_nz^m +\sum_{j=0}^{m-1}A_jz^j\in M_n$. The \textit{spectrum} of $L(z)$ is defined as $\sigma(L)=\{z\in \mathbb C:\det(L(z))=0\}$. Consider the $m$-by-$m$ block companion matrix of $L(z)$. By \cite[Chapter 1]{matpoly}, the spectrum of $L(z)$ is the set of eigenvalues of its companion matrix; in particular, $\sigma(L)$ has at most $mn$ elements. 

\begin{proposition}\label{prop:4}
Let $m,n\in \mathbb N$ where $m\geq 2$ is even. Let $A=\bigoplus_{j=1}^{m-1}A_j\in M_{m-1}(M_n)$ be nonsingular and $B=[B_j]_{j=0}^{m-1}\in M_{1,m}(M_n)$. Then \\
$\textup{nullity}(\textup{Re}(\omega \mathcal{C}_{A,B}))\leq n\ \textup{for all}\ \omega\in \mathbb T\setminus \sigma(P)$ where $P(z):=I_nz^m+\sum_{j=1}^{\frac{m}{2}}P_jz^{2j-2}$ and \[P_j:=
(-1)^{\frac{m}{2}-j}A_{m-1}^{-1}A_{m-2}^*A_{m-3}^{-1}A_{m-4}^*A_{m-5}^{-1}\cdots A_{2j}^*A_{2j-1}^{-1}B_{2j-2}^*\] for all $ j\in \left\{1,2,3,\ldots,\frac{m}{2}\right\}$. 
Moreover, \[\det(\textup{Re}(\omega \mathcal{C}_{A,B}))=\frac{(-1)^{\frac{mn}{2}}}{2^{mn}}\prod_{\substack{j=1\\ j\ \textup{odd}}}^{m-1}|\det(A_j)|^2|\det(P(\omega))|^2\] for all $\omega\in\mathbb T$.
\end{proposition}

\begin{proof} 
Let $\omega\in \mathbb T$. By Proposition \ref{prop:1}, $\textup{Re}(\omega \mathcal{C}_{A,B})$ and $\textup{Re}(\mathcal{C}_{A,Y})$ are unitarily similar where $Y=[\omega^{m-j}B_j]_{j=0}^{m-1}$. 

We first prove the determinant formula. By similarity and Lemma \ref{dettri}, 
\[\det(\textup{Re}(\omega \mathcal{C}_{A,B}))=\det(\textup{Re}( \mathcal{C}_{A,Y}))=\frac{(-1)^{\frac{mn}{2}}}{2^{mn}}\prod_{\substack{j=1\\ j\ \textup{odd}}}^{m-1}|\det(A_j)|^2|\det(\widetilde{P}(1))|^2\]
where 
$\widetilde{P} (z):=I_nz^m+\sum_{j=1}^{\frac{m}{2}}\widetilde{P}_jz^{2j-2}$ and \begin{equation}\label{PTILDE}\widetilde{P}_j:=
(-1)^{\frac{m}{2}-j}A_{m-1}^{-1}A_{m-2}^*A_{m-3}^{-1}A_{m-4}^*A_{m-5}^{-1}\cdots A_{2j}^*A_{2j-1}^{-1}Y_{2j-2}^*\end{equation} for all $ j\in \left\{1,2,3,\ldots,\frac{m}{2}\right\}$. Note that $\widetilde{P}_j=\omega^{-m+2j-2}P_j$ for all $ j\in \left\{1,2,3,\ldots,\frac{m}{2}\right\}$, and so \begin{equation}\label{pomega}\widetilde{P}(1)=I_n+\sum_{j=1}^{\frac{m}{2}}\widetilde{P}_j=I_n+\sum_{j=1}^{\frac{m}{2}}P_j\omega^{-m+2j-2}=\omega^{-m}P(\omega).\end{equation} Since $\omega\in \mathbb T,$ \[\det(\textup{Re}({\cal C}_{A,Y}))=\frac{(-1)^{\frac{mn}{2}}}{2^{mn}}|\det(A_1A_3\cdots A_{m-3}A_{m-1})|^2|\det(P(\omega))|^2.\]

Now we prove the claim about the nullity. By similarity, $\textup{nullity}(\textup{Re}(\omega \mathcal{C}_{A,B}))=\textup{nullity}(\textup{Re}(\mathcal{C}_{A,Y}))$. Consider a vector $x=[x_j]\in M_{m,1}(\mathbb C^n)$ such that $\textup{Re}(\mathcal{C}_{A,Y})x=0.$ We show that $x_m=0$. Inspecting the first $m-1$ block rows of this equation, we conclude that \begin{equation}\label{first}A_1x_2+Y_0^*x_m=0\end{equation} and if $m>2$, \begin{equation}\label{second}A_k^*x_k+A_{k+1}x_{k+2}+Y_k^*x_m=0\ \textup{for all}\ k\in \{1,\ldots, m-2\}.\end{equation} Assume $ \omega\in \mathbb T\setminus \sigma(P)$. If $m=2$, then \eqref{first} implies \[0=\omega^2A_1^{-1}(A_1+\overline{\omega}^2B_0^*)x_2=(\omega^2 I_n+A_1^{-1}B_0^*)x_2=P(\omega) x_2.\] Since $\omega\in \mathbb T\setminus \sigma(P)$, $x_2=0$. In this case, $\textup{nullity}(\textup{Re}(\omega \mathcal{C}_{A,B}))=\textup{nullity}(\textup{Re}( \mathcal{C}_{A,Y}))\leq n$. Now suppose $m>2$ is even. Upon rewriting \eqref{second}, we obtain
\begin{equation}\label{third}x_{k+2}=-A_{k+1}^{-1}A_{k}^*x_{k}-A_{k+1}^{-1}Y_{k}^*x_m\ \textup{for all}\ k\in \{1,\ldots, m-2\}.\end{equation}
By successively applying \eqref{third} and \eqref{PTILDE} and at the last step \eqref{first} and \eqref{PTILDE}, we obtain \[\begin{array}{rcl}x_{m}&=&-A_{m-1}^{-1}A_{m-2}^*x_{m-2}-A_{m-1}^{-1}Y_{m-2}^*x_m\\ &=&-A_{m-1}^{-1}A_{m-2}^*x_{m-2}-\widetilde{P}_{\frac{m}{2}}x_m\\
&=&A_{m-1}^{-1}A_{m-2}^*A_{m-3}^{-1}A_{m-4}^*x_{m-4}-(\widetilde{P}_{\frac{m}{2}-1}+\widetilde{P}_{\frac{m}{2}})x_m\\ &\vdots& \\ &=& -(-1)^{\frac{m}{2}-2}A_{m-1}^{-1}A_{m-2}^*\cdots A_4^*A_3^{-1}A_2^*x_2-\sum_{j=2}^{\frac{m}{2}}\widetilde{P}_jx_m\\ &=&-\sum_{j=1}^{\frac{m}{2}}\widetilde{P}_jx_m\end{array}.\]

\noindent By \eqref{pomega},
\[P(\omega)x_m=\omega^m[\omega^{-m}P(\omega)]x_m=\omega^m\widetilde{P}(1)x_m=\omega^m(I_n+\sum_{j=1}^{\frac{m}{2}}\widetilde{P}_j)x_m=0.\]
 Since $\omega \in \mathbb {T}\setminus\sigma(P)$, $x_m=0$. Hence, the vector $[x_j]_{j=1}^{m-1}\in M_{m-1,1}(\mathbb C^n)$ is in the kernel of the matrix $H$ as defined in the proof of Proposition \ref{prop:2}. Thus, \[\begin{array}{rcl}\textup{nullity}(\textup{Re}(\omega \mathcal{C}_{A,B}))&=&\textup{nullity}(\textup{Re}( \mathcal{C}_{A,Y}))\\ &\leq& \textup{nullity}(H)\\
&=&(m-1)n-\textup{rank}(H)\\
&=&(m-1)n-(m-2)n\\
&=&n. \end{array}\] \end{proof}

\subsection{Zero-dilation index of $\mathcal{C}_{A,B}$}
We are now ready to prove a generalization of \cite[Theorem 3.2]{hG16}.

\begin{theorem}\label{main:zdi}
Let $m,n\in \mathbb N$ where $m\geq 2$. Let $A=\bigoplus_{j=1}^{m-1}A_j\in M_{m-1}(M_n)$ be nonsingular and $B=[B_j]_{j=0}^{m-1}\in M_{1,m}(M_n)$. 
\begin{enumerate}
\item[(i)] If $m\geq 3$ is odd, then $\frac{(m-1)n}{2}\leq d(\mathcal{C}_{A,B})\leq \frac{(m+1)n}{2}.$
\item[(ii)] If $m\geq 2$ is even, then
 $ d(\mathcal{C}_{A,B})=\frac{mn}{2}.$
\end{enumerate}
\end{theorem}
\begin{proof} 
\item[(i)] Assume $m\geq 3$ is odd. The principal submatrix of $\mathcal{C}_{A,B}$ whose rows and columns are indexed by $\{1,2,\ldots,n\}\cup\{2n+1,2n+2,\ldots,3n\}\cup\cdots\cup\{(m-1)n+1,(m-1)n+2,\ldots,mn\}$ is given by $\begin{bmatrix} & & 0_{\frac{(m-1)n}{2}} & &  0\\  B_0& B_2& \cdots& B_{m-3}& B_{m-1}\end{bmatrix}$. From this presentation, $0_{\frac{(m-1)n}{2}}$ is a principal submatrix of $\mathcal{C}_{A,B}$, and so $\frac{(m-1)n}{2}\leq d(\mathcal{C}_{A,B}) $. Now by Proposition \ref{prop:2}, $\textup{nullity}(\textup{Re}(\omega \mathcal{C}_{A,B}))\leq n$ for all $\omega \in \mathbb T.$ Hence, $d(\mathcal{C}_{A,B})\leq \frac{mn+n}{2}=\frac{(m+1)n}{2}$ due to Proposition \ref{approach}.
\item[(ii)] Assume $m\geq 2$ is even. The principal submatrix of $\mathcal{C}_{A,B}$ whose rows and columns are indexed by $\{1,2,\ldots,n\}\cup\{2n+1,2n+2,\ldots,3n\}\cup\cdots\cup\{(m-2)n+1,(m-2)n+2,\ldots,(m-1)n\}$ is $0_{\frac{mn}{2}}$. This proves the lower bound $\frac{mn}{2}\leq d(\mathcal{C}_{A,B})$. To prove that equality holds, we let $P(z)$ be as defined in Proposition \ref{prop:4}. For a given $\omega \in \mathbb T$, $\textup{Re}(\omega {\cal C}_{A,B})$ is singular if and only if $\omega \in \sigma(P)$ due to Proposition \ref{prop:4}. Since $\mathbb T\cap \sigma(P)\subset \sigma(P)$ has at most finitely many elements (see \cite[Chapter 1]{matpoly}), there exists $\zeta\in \mathbb T$ for which $\textup{Re}(\zeta \mathcal{C}_{A,B})$ is nonsingular. By Proposition \ref{approach}, $d(\mathcal{C}_{A,B})\leq \frac{mn}{2}.$ Therefore, $d(\mathcal{C}_{A,B})=\frac{mn}{2}.$

\end{proof}


In the following example, we show that it is possible to get as zero-dilation index each integer value between $\frac{(m-1)n}{2} $ and $\frac{(m+1)n}{2}$ in Theorem \ref{main:zdi}(i).

\begin{example}\normalfont
Let $m,n\in \mathbb N$ where $m\geq 3$ is odd. Let $A=\oplus_{j=1}^{m-1}I_n=I_{(m-1)n}$ and $B=\begin{bmatrix} 0& \cdots& 0& H/2\end{bmatrix}\in M_{1,m}(M_n)$ where $H\in M_n$ is a Hermitian matrix to be determined later. By Proposition \ref{prop:1} and definition of $B$, $\textup{Re}(\omega \mathcal{C}_{A,B})$ is unitarily similar to $\textup{Re}(\mathcal{C}_{A,\omega B})$ for any $\omega \in \mathbb T$. Hence, \cite[Theorem 2.2]{hG14b} implies \[d(\mathcal{C}_{A,B})=\min\{i_{\geq 0}(\textup{Re}(\mathcal{C}_{A,\omega B})):\omega\in\mathbb{T}\}=\min\{i_{\geq 0}(M_\theta):\theta\in[0,2\pi)\}\] where $M_\theta:=2\textup{Re}(\mathcal{C}_{A,e^{i\theta} B})=[J_{m}(0)+J_m(0)^*]\otimes I_n+ (0_{m-1}\oplus [1])\otimes (\cos\theta  H)$ for any $\theta\in [0,2\pi)$. We partition $M_\theta$ so that an application of Haynsworth's Theorem (see \cite[Problem 4.5.P21]{HJ1}) allows us to solve $i_{\geq 0}(M_\theta ).$ Let $N=[J_{m-1}(0)+J_{m-1}(0)^*]\otimes I_n$ which is nonsingular since $m$ is odd by assumption. Direct calculations reveal that the Schur complement of $N$ in $M_\theta$ is equal to $\cos\theta H$. By Haynsworth's Theorem, \[i_{\geq 0}(M_\theta)=i_+(N)+i_{\geq 0}(\cos\theta H)=\begin{cases}i_+(N)+ n,& \textup{if}\ \theta=\frac{\pi}{2}\ \textup{or}\ \theta=\frac{3\pi}{2}\\ i_+(N)+i_{\geq 0}(H),& \textup{if}\ \theta\in \left[0,\frac{\pi}{2}\right)\cup \left(\frac{3\pi}{2},2\pi\right)\\ i_+(N)+i_{\leq 0}(H),& \textup{if}\ \theta\in \left(\frac{\pi}{2},\frac{3\pi}{2}\right).\end{cases}\] It can also be computed that $i_{+}(N)=\frac{(m-1)n}{2}$. Hence, 
\[d(\mathcal{C}_{A,B})=\frac{(m-1)n}{2}+\textup{min}\{n, i_{\geq 0}(H),i_{\leq 0}(H)\}.\]


\noindent If $H$ is chosen to be $H=0_k\oplus-I_{n-k}$ where $k\in\{0,\ldots,n\}$ (a direct summand is absent if its subscript is $0$), then 
\[d(\mathcal{C}_{A,B})=\frac{(m-1)n}{2}+\min\{k,n\}=\frac{(m-1)n}{2}+k.\]
It follows that $d(\mathcal{C}_{A,B})$ returns all integer values between $\frac{(m-1)n}{2}$ and $\frac{(m+1)n}{2}.$
\end{example}




\section{Block KMS matrices}





In this section, we study properties of block KMS matrices. 

\subsection{Similarity or unitary similarity of block KMS matrices}
For $m=2$, the similarity or unitary similarity of the block KMS matrices is determined by the the rank or singular values of the matrix parameters.

\begin{proposition}
Let $A,B\in M_n$. 
\begin{enumerate}
    \item[(i)] $\mathcal{K}_2(A)$ and $\mathcal{K}_2(B)$ are similar if and only if $\mbox{rank}(A)=\mbox{rank}(B).$
    \item[(ii)] $\mathcal{K}_2(A)$ and $\mathcal{K}_2(B)$ are unitarily similar if and only if $A$ and $B$ have the same singular values.
\end{enumerate}
\end{proposition}
\begin{proof}
\item[(i)] If $\mathcal{K}_2(A)$ and $\mathcal{K}_2(B)$ are similar, then \[\mbox{rank}(A)=\mbox{rank}(\mathcal{K}_2(A))=\mbox{rank}(\mathcal{K}_2(B))=\mbox{rank}(B).\] Conversely, if $A$ and $B$ have equal ranks, then there exist nonsingular $P,Q\in M_n  $ such that $PAQ=B$ \cite[Section 0.4.6]{HJ1}. Hence, the following similarity holds
\[(P\oplus Q^{-1})\mathcal{K}_2(A) (P\oplus Q^{-1})^{-1}=\begin{bmatrix}0& PAQ\\ 0& 0\end{bmatrix}=\mathcal{K}_2(B).\]

\item[(ii)] If $U^*\mathcal{K}_2(A)U=\mathcal{K}_2(B)$ for some unitary $U\in M_2(M_n  )$, then \[U^*(AA^*\oplus 0_n)U=U^*\mathcal{K}_2(A)\mathcal{K}_2(A)^*U=\mathcal{K}_2(B)\mathcal{K}_2(B)^*=BB^*\oplus 0_n.\] Hence, $A$ and $B$ have the same singular values. Conversely, if $A$ and $B$ have the same singular values, then there exist unitary $U,V\in M_n  $ such that $UAV=B$. Note that $U\oplus V^*$ is unitary and the following unitary similarity holds
\[(U\oplus V^{*})\mathcal{K}_2(A) (U\oplus V^{*})^{*}=\begin{bmatrix}0& UAV\\ 0& 0\end{bmatrix}=\mathcal{K}_2(B).\]\end{proof}



The next example illustrates that the rank and singular values of the matrix parameters are not enough to characterize similarity and unitary similarity of block KMS matrices when $m\geq 3$. 

\begin{example}\label{exrank}
\normalfont
Let $m\geq 3$. Consider  $A=J_2(0)$ and $B=\textup{diag}(1,0).$ Note that $A$ and $B$ have the same singular values; hence, they necessarily have equal ranks. For any $k\geq 2$, $A^k=0$ while $B^k=B$. Direct calculations reveal that  $(\mathcal{K}_m(A))^2=0$ while $(\mathcal{K}_m(B))^2\neq 0$. It follows that $\mathcal{K}_m(A)$ and $\mathcal{K}_m(B) $ are not similar and not unitarily similar as well.
\end{example}

The next part is devoted to determining the complete set of conditions for the matrix parameters that characterize the similarity and unitary similarity of block KMS matrices. 



\begin{proposition}\label{S}
Let $m,n\in \mathbb N$ where $m\geq 2$ and let $A\in M_n$. Then $\mathcal{K}_m(A)=(I_{mn}-J_m(0)\otimes A)^{-1}-I_{mn}$.
\end{proposition}
\begin{proof} For convenience, let $S_r:=I_{rn}-J_r(0)\otimes A$ for any $r\geq 2$. We consider the equivalent claim $(I_{mn}+\mathcal{K}_m(A))S_m=I_{mn}$ and prove it by induction on $m$. If $m=2$, then 
\[ (I_{2n}+\mathcal{K}_2(A))S_2=\begin{bmatrix} I_n& A\\ 0& I_n\end{bmatrix}\begin{bmatrix}I_n& -A\\ 0& I_n\end{bmatrix}=\begin{bmatrix}I_n& 0\\ 0& I_n\end{bmatrix}=I_{2n}.\] Suppose $m\geq 3$ and the statement holds for block matrices of size at most $m-1$. Let \begin{equation}\label{KL}K=\begin{bmatrix} A^{m-1}\\ \vdots\\ A^2\\ A\end{bmatrix},\ L=\begin{bmatrix}  0\\ \vdots\\ 0\\ -A\end{bmatrix}\in M_{m-1,1}(M_n  ).\end{equation} Then
\[\begin{array}{rcl}(I_{mn}+\mathcal{K}_m(A))S_m&=&\begin{bmatrix}I_{(m-1)n}+\mathcal{K}_{m-1}(A) & K\\ 0& I_n\end{bmatrix}\begin{bmatrix} S_{m-1} & L\\ 0& I_n\end{bmatrix}\\
&=&\begin{bmatrix} (I_{(m-1)n}+\mathcal{K}_{m-1}(A))S_{m-1}& M\\ 0& I_n\end{bmatrix}\end{array}\]
where the claim follows due to the inductive hypothesis and the calculation $M=(I_{(m-1)n}+\mathcal{K}_{m-1}(A))L+K=0$.
\end{proof}

\begin{proposition}\label{kron}
Let $m,n\in \mathbb N$ where $m\geq 2$ and let $A,B\in M_n  $. Then $\mathcal{K}_m(A)$ and $\mathcal{K}_m(B)$ are similar (unitarily similar) if and only if $J_m(0)\otimes A$ and $J_m(0)\otimes B$ are similar (unitarily similar). 
\end{proposition}
\begin{proof}
Suppose $P^{-1}\mathcal{K}_m(A)P=\mathcal{K}_m(B)$ for some nonsingular (unitary) $P\in M_n  $. By Proposition \ref{S}, 
\[\begin{array}{rcl}
J_m(0)\otimes B&=&I_{mn}-(I_{mn}+\mathcal{K}_m(B))^{-1}\\
&=&I_{mn}-(I_{mn}+P^{-1}\mathcal{K}_m(A)P)^{-1}\\
&=&P^{-1}(I_{mn}-[I_{mn}+\mathcal{K}_m(A))^{-1}]P\\
&=&P^{-1}(J_m(0)\otimes A)P.
\end{array}\]
Hence, $J_m(0)\otimes A$ and $J_m(0)\otimes B$ are similar (unitarily similar).

Conversely, suppose $J_m(0)\otimes B=P^{-1}(J_m(0)\otimes A) P$ for some nonsingular (unitary) $P\in M_n  $. By Proposition \ref{S}, 
\[\begin{array}{rcl} \mathcal{K}_m(B)&=&(I_{mn}-J_m(0)\otimes B)^{-1}-I_{mn}\\
&=&(I_{mn}-P^{-1}(J_m(0)\otimes A)P)^{-1}-I_{mn}\\
&=&P^{-1}((I_{mn}-J_m(0)\otimes A)^{-1}-I_{mn})P\\
&=&P^{-1}\mathcal{K}_m(A)P.\\
\end{array}\]
It follows that $\mathcal{K}_m(A)$ and $\mathcal{K}_m(B)$ are similar (unitarily similar).
\end{proof}



For each $\lambda\in\sigma(A)$, let $w_1(A,\lambda):=\mbox{nullity}(A-\lambda I_n)$ and define $s_k (A, \lambda )= 0$ for all $k > w_1(A, \lambda)$ and let \[s_1(A,\lambda)\geq s_2(A,\lambda) \geq\cdots\geq s_{w_1(A,\lambda)}(A,\lambda)>0\] be the nonincreasingly ordered list of sizes of Jordan blocks associated to $\lambda$ (this is called the Segre characteristic of $A$ associated with $\lambda$). The Jordan Canonical Form (JCF) Theorem \cite[Theorem 3.1.11]{HJ1} guarantees that the Segre characteristic for each eigenvalue completely determines the Jordan canonical form of $A$.

Let $m\in \mathbb N$ with $m\geq 2$ and $A\in M_n  $. By Proposition \ref{kron}, it suffices to characterize the JCF of $J_m(0)\otimes A$. Since $(J_m(0)\otimes A)^m=0$, the JCF of $J_m(0)\otimes A$ contains only nilpotent Jordan blocks of size at most $m$. To determine this nilpotent structure, it is convenient to recall the indicator function notation: for a given $\mathcal{S}\subseteq \mathbb{R},$ let $\mathbf{1}_{\mathcal S}(x)=\begin{cases}1,& \textup{if}\ x\in \mathcal{S}\\ 0,& \textup{otherwise}\end{cases}.$ 


\begin{proposition}\label{FGcount}
Let $m,n\in \mathbb N$ where $m\geq 2$ and let $A\in M_n  $. If $0\in \sigma(A)$, let $\alpha$ be the algebraic multiplicity of $0$; otherwise, $\alpha:=0$. For each $k\in\{1,\ldots,m\}$, let $N_k^A$ be the number of $J_k(0)$ blocks in the JCF of $J_m(0)\otimes A.$ If $k=m$, then \[N_m^A=n-\alpha+\sum_{j=1}^{w_1(A,0)}(|m-s_j(A,0)|+1)\cdot \mathbf{1}_{[1, s_j(A,0)]}(m).\] If $k<m$, then \[N_k^A=\sum_{j=1}^{w_1(A,0)}[2\cdot \mathbf{1}_{[1,\min\{m,s_j(A,0)\})}(k)+(|m-s_j(A,0)|+1)\cdot \mathbf{1}_{\{\min\{m,s_j(A,0)\}\}}(k)].\]

\end{proposition}
\begin{proof}

Let $\lambda\in \sigma(A)$. By (iii) and (iv) of \cite[Theorem 4.3.17]{HJ2}, two cases arise depending on $\lambda$. If $\lambda \neq 0$, then the JCF of $J_m(0)\otimes A$ has $q$ copies of $J_m(0)$ for each $q\in\{s_1(A,\lambda),\ldots,s_{w_1(A,\lambda)}(A,\lambda)\}$. Otherwise, if $\lambda=0$ and $q\in\{s_1(A,0),\ldots,s_{w_1(A,0)}(A,0)\}$, then the JCF of $J_m(0)\otimes A$ has two copies of $J_k(0)$ for each $k\in\{1,\ldots,\min\{m,q\}-1\}$ (absent if $\min\{m,q\}=1$) and $|m-q|+1$ copies of $J_{\min\{m,q\}}(0)$. 

Assume $k=m$. By case (iii) of \cite[Theorem 4.3.17]{HJ2}, all nonzero eigenvalues of $A$ contributes $\displaystyle\sum_{\begin{subarray}{c} \lambda\in \sigma(A)\\  \lambda\neq 0\end{subarray}}\sum_{j=1}^{w_1(A,\lambda)}s_j(A,\lambda)=n-\alpha $ to $N_m^A$. By case (iv) of \cite[Theorem 4.3.17]{HJ2}, each $s_j(A,0)$ where $m\leq s_j(A,0)$ contributes $|m-s_j(A,0)|+1.$ This proves the formula for $N_m^A.$

Assume $k<m.$ By case (iv) of \cite[Theorem 4.3.17]{HJ2}, each $s_j(A,0)$ contributes $2$ to provided $k\leq \min\{m,s_j(A,0)\}-1$ and $|m-s_j(A,0)|+1$ provided $k=\min\{m,s_j(A,0)\}.$ This proves the formula for $N_k^A.$
\end{proof}

\begin{example}
\normalfont
Let $m\geq 2$ and $A=J_3(0)\oplus[0]$. Note that $s_1(A,0)=3$ and $s_2(A,0)=1$. Assume $m\in \{2,3\}$. By Proposition \ref{FGcount}, $N_m^A=4-4+|m-s_1(A,0)|+1+0=4-m$. Moreover, if $m=2$, then $N_1^A=4$ while if $m=3,$ then $N_1^A=5$ and $N_2^A=2$. Assume $m>3$. Then $N_m^A=0$. Direct calculations reveal that $N_1^A=m+2$, $N_2^A=2$, $N_3^A=m-2$, and $N_k^A=0$ for $k\in \{4,\ldots, m-1\}.$ 
\end{example}

Knowing $N_k^A$ completely determines the Segre characteristic of $J_m(0)\otimes A$. By Proposition \ref{kron}, we have the following characterization of similarity of block KMS matrices.

\begin{theorem}\label{sim}
Let $m,n\in \mathbb N$ where $m\geq 2$ and let $A,B\in M_n  $. Then $\mathcal{K}_m(A)$ and $\mathcal{K}_m(B)$ are similar if and only if $N_k^A=N_k^B $ for each $k\in\{1,\ldots,m\}.$ 
\end{theorem}

\begin{example}\normalfont
Let $A=J_2(0)$ and $B=\textup{diag}(1,0).$ If $m\geq 3$, then $N_m^A=0$ while $N_m^B=1$. By Theorem \ref{sim}, $\mathcal{K}_m(A)$ and $\mathcal{K}_m(B)$ are not similar. This is not surprising due to Example \ref{exrank}.
\end{example}

The next result has already been observed for $n=1$ in \cite[Proposition 2.1(b)]{hG13}.

\begin{corollary}
    Let $m,n\in \mathbb N$ where $m\geq 2$ and let $A,B\in M_n  $ be nonsingular. Then $\mathcal{K}_m(A)$ and $\mathcal{K}_m(B)$ are similar.
    
\end{corollary}
\begin{proof}


Since $A,B$ are nonsingular, $N_m^A=n=N_m^B$ and $N_k^A=0=N_k^B$ for each $k\in\{1,\ldots,m-1\}$. The claim follows from Theorem \ref{sim}.    \end{proof}




Let $m\in \mathbb N$ with $m\geq 2$ and $A,B\in M_n  .$ By Proposition \ref{kron} and Specht's Theorem \cite[Theorem 2.2.6]{HJ1}, $\mathcal{K}_m(A)$ and $\mathcal{K}_m(B)$ are unitarily similar if and only if \[\mbox{tr}[p(J_m(0)\otimes A,(J_m(0)\otimes A)^*)]=\mbox{tr}[p(J_m(0)\otimes B,(J_m(0)\otimes B)^*)]\] for every word $p(s,t)$ in two noncommuting variables $s,t$. Let $p(s,t)=s^{m_1}t^{n_1}s^{m_2}t^{n_2}\cdots s^{m_k}t^{n_k}$ with $m_1,n_1,\ldots,m_k,n_k\in \mathbb{N}\cup\{0\}$. Properties of Kronecker products \cite[Section 4.2]{HJ2} imply that
\[\begin{array}{rcl}p(J_m(0)\otimes Z,(J_m(0)\otimes Z)^*)&=&\displaystyle\prod_{j=1}^k(J_m(0)\otimes Z)^{m_j}((J_m(0)\otimes Z)^*)^{n_j}\\
&=&\displaystyle\prod_{j=1}^k(J_m(0)^{m_j}\otimes Z^{m_j})((J_m(0)^*)^{n_j}\otimes (Z^* )^{n_j})\\
&=&\displaystyle\prod_{j=1}^kJ_m(0)^{m_j}(J_m(0)^*)^{n_j}\otimes \displaystyle\prod_{j=1}^k Z^{m_j}(Z^* )^{n_j}\\
&=&p(J_m(0),J_m(0)^*)\otimes p(Z,Z^*).\end{array}\]
for any square matrix $Z$. By the trace property \cite[Section 4.2, Problem 12]{HJ2}, it follows that 
\[\textup{tr}[p(J_m(0)\otimes Z,(J_m(0)\otimes Z)^*)]=\textup{tr}[p(J_m(0),J_m(0)^*)]\cdot \textup{tr}[p(Z,Z^*)]\]
Hence, it suffices to check the trace condition on words $p(s,t)$ for which $\textup{tr}[p(J_m(0),J_m(0)^*)]\neq 0$.
\begin{theorem}\label{usim}
Let $m,n\in \mathbb N$ where $m\geq 2$ and let $A,B\in M_n  $. Then $\mathcal{K}_m(A)$ and $\mathcal{K}_m(B)$ are unitarily similar if and only if $\textup{tr}[p(A,A^*)]=\textup{tr}[p(B,B^*)]$ for every word $p(s,t)$ in noncommuting variables $s,t$ for which $\textup{tr}[p(J_m(0),J_m(0)^*)]\neq 0$. 
\end{theorem}

\begin{corollary}
    Let $m,n\in \mathbb N$ where $m\geq 2$ and let $A,B\in M_n$ be normal. Let $\lambda_j(A)$ and $\lambda_j(B)$ denote the eigenvalues of $A$ and $B$, respectively. Then $\mathcal{K}_m(A)$ and $\mathcal{K}_m(B)$ are unitarily similar if and only if $\displaystyle\sum_{j=1}^np(\lambda_j(A),\overline{\lambda_j(A)})=\displaystyle\sum_{j=1}^np(\lambda_j(B),\overline{\lambda_j(B)})$ for every word $p(s,t)$ in noncommuting variables $s,t$ for which $\textup{tr}[p(J_m(0),J_m(0)^*)]\neq 0$. In particular, if $a,b\in \mathbb C$, then $\mathcal{K}_m(a)$ and $\mathcal{K}_m(b)$ are unitarily similar if and only if $|a|=|b|$.
\end{corollary}
\begin{proof}
By the Spectral Theorem, $A=UDU^*$ where $D=\textup{diag}(\lambda_1(A),\ldots,\lambda_n(A))$ and $U\in M_n$ is unitary. Let $p(s,t)$ be a word in noncommuting variables $s,t$. Then $p(A,A^*)=Up(D,D^*)U^*$. Hence, $\mbox{tr}[p(A,A^*)]=\mbox{tr}[p(D,D^*)]=\displaystyle\sum_{j=1}^np(\lambda_j(A),\overline{\lambda_j(A)})$. Similarly, $\mbox{tr}[p(B,B^*)]=\displaystyle\sum_{j=1}^np(\lambda_j(B),\overline{\lambda_j(B)}).$ The equivalence is established by applying Theorem \ref{usim}.

For the remaining claim, if $|a|=|b|$, then $\mathcal{K}_m(a) $ and $\mathcal{K}_m(b)$ are unitarily similar due to \cite[Proposition 2.1(a)]{hG13}. Conversely, assume $\mathcal{K}_m(a)$ and $\mathcal{K}_m(b)$ are unitarily similar. Note that for the word $p(s,t)=st$, $\mbox{tr}[p(J_m(0),J_m(0)^*)]=m-1\neq0$. By Theorem \ref{usim}, $|a|^2=a\overline{a}=p(a,\overline{a})=p(b,\overline{b})=b\overline{b}=|b|^2$, and so $|a|=|b|.$
\end{proof}



\subsection{Zero-dilation index of $\mathcal{K}_m(A)$}

We first compute the zero-dilation index of $\mathcal{K}_2(A)$.

\begin{theorem}\label{thm:zdikms2by2}
Let $n\in \mathbb N$ and let $A\in M_n  $. Then $d(\mathcal{K}_2(A))=n+\textup{nullity}(A).$
\end{theorem}
\begin{proof}
Let $\sigma_1\geq \cdots\geq \sigma_n\geq 0$ be the singular values of $A$. For any $\theta\in \mathbb  R,$ $\textup{Re}(e^{i\theta}\mathcal{K}_2(A) )$ is unitarily similar to $\bigoplus_{i=1}^n\textup{Re}(\sigma_i J_2(0))$. Since $\sigma_1,\ldots,\sigma_n\geq 0$, $i_{\geq 0}(\textup{Re}(e^{i\theta}\mathcal{K}_2(A) ))=i_{\geq 0}(\bigoplus_{i=1}^n\textup{Re}(\sigma_i J_2(0)))=n+i_0(A).$ The claim follows due to \cite[Theorem 2.2]{hG14b}.
\end{proof}

For $m\geq 3$, the computation of the zero-dilation index of $\mathcal{K}_m(A)$ requires a series of observations.

\begin{lemma}\label{*cong}
Let $m,n\in \mathbb N$ where $m\geq 2$ and let $A\in M_n  $. There exists unit upper triangular $S\in M_{mn}  $ such that
 $S[\alpha \mathcal{K}_m(A)+\beta \mathcal{K}_m(A)^*]S^*=\alpha J_m(0)\otimes A+\beta [J_m(0)\otimes A]^*-(\alpha+\beta) (I_{m-1}\oplus [0])\otimes AA^*$ for any $\alpha,\beta\in \mathbb C.$ In addition, if $\beta=\alpha^{-1}$ ($\beta=\overline{\alpha}\in \mathbb{T}$), then there exists nonsingular (unitary) $T\in M_{mn}   $ such that $TS[\alpha \mathcal{K}_m(A)+\beta \mathcal{K}_m(A)^*]S^*T^{-1}= 2\textup{Re}(J_m(0)\otimes A)-(\alpha+\beta) (I_{m-1}\oplus [0])\otimes AA^*$.  
 \end{lemma}
\begin{proof}
Set $S:=I_{mn}-J_r(0)\otimes A$. Proposition \ref{S} implies the identity $S\mathcal{K}_m(A)=I_{mn}-S$. Using this, we have $S(\alpha \mathcal{K}_m(A)+\beta\mathcal{K}_m(A)^*)S^*=\alpha J_m(0)\otimes A+\beta [J_m(0)\otimes A]^*-(\alpha+\beta) (I_{m-1}\oplus [0])\otimes AA^*$. The remaining claim follows by taking $T=\textup{diag}(1,\alpha,\ldots,\alpha^{m-1})\otimes I_n$.
\end{proof}

\begin{lemma}\label{inc}
Let $m,n\in \mathbb N$ where $m\geq 2$ and let $A\in M_n  $. For $k\in\{1,\ldots,m\}$ and for any $\theta\in \mathbb R$, define
\[X_k^A(\theta):=\begin{cases}2\textup{Re}(J_m(0)\otimes A)-2\cos\theta (I_{m-1}\oplus[0])\otimes AA^*,&\ \textup{if}\ k=m\\ 
2\textup{Re}(J_k(0)\otimes A)-2\cos\theta (I_k\otimes AA^*),&\ \textup{if}\ k<m.\end{cases}\]
If $0\leq \theta_1\leq \theta_2\leq \pi$, then $i_{\geq 0}(X_k^A(\theta_1))\leq i_{\geq 0}(X_k^A(\theta_2)).$ If $\pi\leq \theta_1\leq \theta_2\leq 2\pi$, then $i_{\geq 0}(X_k^A(\theta_1))\geq i_{\geq 0}(X_k^A(\theta_2))$.
\end{lemma}
\begin{proof}
For each $k\in\{1,\ldots,m\}$, let \[D_k:=\begin{cases} I_{m-1}\oplus [0],& \textup{if}\ k=m\\ I_k,& \textup{if}\ k<m.\end{cases}\] For any $\theta_1,\theta_2\in \mathbb R$, observe that
\[X_k^A(\theta_2)-X_k^A(\theta_1)=2(\cos\theta_1-\cos\theta_2)(D_k\otimes AA^*).\]
Since $\cos\theta$ is decreasing on $[0,\pi]$, $X_k^A(\theta_2)-X_k^A(\theta_1)\geq 0$ for any $0\leq \theta_1\leq \theta_2\leq \pi.$ The first claim follows by applying Weyl's inequality (see \cite[Theorem 4.3.1]{HJ1} or \cite[Problem 4.4.P4]{HJ1}). Similarly, the second claim follows from Weyl's inequality and so $X_k^A(\theta_2)-X_k^A(\theta_1)\leq 0$ for any $\pi\leq \theta_1\leq \theta_2\leq 2\pi$ since $\cos\theta $ is increasing on $[\pi,2\pi].$
\end{proof}

\begin{theorem}\label{thm:zdikms}
Let $m,n\in \mathbb N$ where $m\geq 3$ and let $A\in M_n  $. Then $d(\mathcal{K}_m(A))=i_{\geq 0}(\textup{Re}(\mathcal{K}_m(A)))=i_{\geq 0}(2\textup{Re}(J_m(0)\otimes A)- 2(I_{m-1}\oplus[0])\otimes AA^*).$ If $A$ is nonsingular, then \[d(\mathcal{K}_m(A))=
 n+i_{\geq 0}(2\textup{Re}(J_{m-2}(0)\otimes A)- 2(I_{m-2}\otimes AA^*)).\]
\end{theorem}
\begin{proof}
By \cite[Theorem 2.2]{hG14b}, $d(\mathcal{K}_m(A))=\min\{i_{\geq 0}(\textup{Re}(e^{i\theta}\mathcal{K}_m(A))):\theta \in \mathbb R\}$. By Lemma \ref{*cong} and Sylvester's law of inertia \cite[Theorem 4.5.8]{HJ1}, \[i_\geq (\textup{Re}(e^{i\theta}\mathcal{K}_m(A)))=i_{\geq 0}(X_m^A(\theta)),\ \textup{for any}\ \theta\in \mathbb R.\]
Hence, the monotonicity from Lemma \ref{inc} implies \[\begin{array}{rcl}d(\mathcal{K}_m(A))&=&\min\{i_{\geq 0}(\textup{Re}(e^{i\theta}\mathcal{K}_m(A))):\theta \in \mathbb R\}\\
&=&\min\{i_{\geq 0}(X_m^A(\theta )):\theta \in [0,\pi]\cup[\pi,2\pi]\}\\
&=&\min\{i_{\geq 0}(X_m^A(0)),i_{\geq 0}(X_m^A(2\pi))\}\\
&=&i_{\geq 0}(\textup{Re}(\mathcal{K}_m(A))).\end{array}\]

To prove the remaining claim, assume $A$ is nonsingular. Consider $B:=X_2^A(0)=\begin{bmatrix} -2 AA^* & A\\ A^*& 0\end{bmatrix}$ and $C:=X_2^{[1]}(0)=\begin{bmatrix} -2 & 1\\ 1& 0\end{bmatrix}$ (which are both nonsingular). Note that $B=(A\oplus I_n)\left(C\otimes I_n\right)(A\oplus I_n)^*.$ By Sylvester's law of inertia and the eigenvalues of $C$, $i_+(B)=n\cdot i_+\left(C\right)=n$. Moreover, $B$ is a nonsingular principal submatrix of $X_m^A(0)$. Direct computations reveal that the Schur complement of $B$ in $X_m^A(0)$ is equal to $X_{m-2}^A(0)$. The claim follows from Haynsworth's Theorem (see \cite[Problem 4.5.P21]{HJ1}): \[i_{\geq 0}(\textup{Re}(\mathcal{K}_m(A)))=i_+(B)+i_{\geq 0}(X_{m-2}^A(0))=n+i_{\geq 0}(X_{m-2}^A(0)).\] \end{proof}

In general, $i_{\geq 0}(\textup{Re}(\mathcal{K}_m(A)))$ or $i_{\geq 0}(X_{m-2}^A(0))$ might still be difficult to compute directly. In the next result, adding normality to the assumptions on $A$ yields a formula for $d(\mathcal{K}_m(A))$.

\begin{corollary}\label{cor:zdikms}
Let $m,n\in \mathbb N$ where $m\geq 3$ and let $A\in M_n  $ be nonsingular and normal. Let $\lambda_1,\ldots,\lambda_n$ be the eigenvalues of $A$. For each $i\in\{1,\ldots,n\}$, define $k_i$ to be the largest element of $\{1,\ldots, m-2\}$ for which $\cos\left(\frac{k_i\pi }{m-1}\right)<|\lambda_i|\leq \cos\left[\frac{(k_i-1)\pi }{m-1}\right]$; if no such element of $\{1,\ldots,m-2\}$ exists, define $k_i:=1$. Then $d(\mathcal{K}_m(A))=k_1+\ldots+k_n.$
\end{corollary}
\begin{proof}
Spectral Theorem guarantees that $2\textup{Re}(J_{m-2}(0)\otimes A)- 2(I_{m-2}\otimes AA^*)$ is unitarily similar to $\bigoplus_{i=1}^n[2\textup{Re}(\lambda_iJ_{m-2}(0))- 2|\lambda_i|^2I_{m-2}].$
The eigenvalues of each direct summand are $-2|\lambda_i|^2+|\lambda_i|\cos \frac{k\pi}{m-1}$ where $k\in \{1,\ldots,m-2\}$ (this is because the direct summands are tridiagonal \cite[Problem 1.4.P16]{HJ1}). Note that $-2|\lambda_i|^2+|\lambda_i|\cos \frac{k\pi}{m-1}\geq 0$ if and only if $\cos\frac{k\pi}{m-1}\geq |\lambda_i|.$ By definition of $k_i$ and Theorem \ref{thm:zdikms}, 
\[d(\mathcal{K}_m(A))=n+(k_1-1)+\cdots+(k_n-1)=k_1+\cdots+k_n.\]\end{proof}

Theorem \ref{thm:zdikms} and Corollary \ref{cor:zdikms} are generalizations of \cite[Theorem 2.1]{hG14a} for nonsingular $A$.

\subsection{Circularity of the numerical range of $\mathcal{K}_m(A)$}


The Kippenhahn polynomial of $A\in M_m$ is the homogeneous polynomial \[p_A(x,y,z)=\textup{det}(x\textup{Re}(A)+y\textup{Im}(A)+z I_m).\] The convex hull of the real points of the dual curve of $p_A(x,y,z)=0$ is $W(A)$ \cite[Theorem 10]{kipp}.

\begin{lemma}\label{lemdet}
Let $m,n\in \mathbb N$ where $m\geq 3$ and let $A\in M_n  $. Then $\det[\alpha \mathcal{K}_m(A)+\beta \mathcal{K}_m(A)^*]=(-\alpha\beta)^n \det(A^*A)\det[\alpha J_{m-2}(0)\otimes (A)+\beta [J_{m-2}(0)\otimes A]^*-(\alpha+\beta) I_{m-2}\otimes AA^*]$. In particular, \[\label{det}p_{\mathcal{K}_m(A)}(1,y,0)=\dfrac{(-1)^n(1+y^2)^n\det(A^*A)}{4^n}h(y)\]
where $h(y):=\det\left[\frac{1-iy}{2} J_{m-2}(0)\otimes (A)+\frac{1+iy}{2} [J_{m-2}(0)\otimes A]^*- I_{m-2}\otimes AA^*\right]$.
\end{lemma}

\begin{proof}
Let $\alpha,\beta\in\mathbb C$. By Lemma \ref{*cong}, there exists unit upper triangular $S\in M_{mn}  $ such that $S(\alpha \mathcal{K}_m(A)+\beta\mathcal{K}_m(A)^*)S^*$ is equal to \begin{equation}\label{SS*}\begin{bmatrix} -(\alpha+\beta) AA^*& \alpha A& 0& \cdots & 0\\ 
\beta A^* & -(\alpha+\beta)AA^*& \ddots& \ddots & \vdots\\ 
0 & \beta A^*& \ddots& \alpha A  &0\\ 
\vdots & \ddots &  \ddots &-(\alpha+\beta)AA^*& \alpha A \\ 0 & \cdots & 0& \beta A^*& 0\end{bmatrix}.\end{equation}
To prove the first claim, interchange the last two block columns of \eqref{SS*} and then take the determinant.

For the second claim, note that \[p_{\mathcal{K}_m(A)}(1,y,0) = 
\det[\textup{Re}(\mathcal{K}_m(A))+y\textup{Im}(\mathcal{K}_m(A))]=\det[\alpha\mathcal{K}_m(A)+\beta\mathcal{K}_m(A)^*]\] where $\alpha=\frac{1}{2}(1-iy)$ and $\beta=\frac{1}{2}(1+iy).$ Applying the first claim completes the proof.\end{proof}

\begin{lemma}\label{singular}
Let $m,n\in \mathbb N$ where $m\geq 3$ and let $A\in M_n  $. If the boundary of $W(\mathcal{K}_m(A))$ contains an elliptic arc, then $A$ is singular.
\end{lemma}
\begin{proof}
Let $E $ be an elliptic disk such that the boundary of $\mathcal{K}_m(A)$ contains an arc of $E$. There exists $B_0\in M_2  $ such that $W(B_0)=E$. By duality and Bezout's Theorem \cite[Theorem 3.9]{fK92}, $p_{B_0}(x,y,z)$ divides $p_{\mathcal{K}_m(A)}(x,y,z).$ In particular, $p_{B_0}(1,i,z)=\det(B_0+zI_n)$ divides $p_{\mathcal{K}_m(A)}(1,i,z)=\det(\mathcal{K}_m(A)+zI_n)=z^{mn}.$ Hence, all eigenvalues of $B_0$ are zero. By a unitary transformation, assume without loss of generality that $B_0=\mathcal{K}_2(2b)=\begin{bmatrix} 0& 2b\\ 0& 0\end{bmatrix}$ for some $b>0.$ Consider $B_0\otimes I_n=\mathcal{K}_2(2bI_n)$. Then $W(\mathcal{K}_2(2bI_n))=\{z\in \mathbb C: |z|\leq b\}=W(B_0)=E.$ By duality and Bezout's Theorem, $p_{\mathcal{K}_2(2bI_n)}(x,y,z)$ divides $p_{\mathcal{K}_m(A)}(x,y,z)$. Let $q(x,y,z):=\dfrac{p_{\mathcal{K}_m(A)}(x,y,z)}{p_{\mathcal{K}_2(2bI_n)}(x,y,z)}$. By properties of determinants of Kronecker products, \[p_{\mathcal{K}_2(2b I_n)}(x,y,z)=\det\left( \begin{bmatrix} z & (x-iy)b\\ (x+iy)b& z\end{bmatrix}\otimes I_n\right)=(z^2-b^2(x^2+y^2))^n.\] 
Lemma \ref{lemdet} implies that 
\[ q(1,y,0)=\dfrac{\det(A^*A)}{4^nb^{2n}}h(y)\] where $h(y):=\det\left[\frac{1-iy}{2} J_{m-2}(0)\otimes (A)+\frac{1+iy}{2} [J_{m-2}(0)\otimes A]^*- I_{m-2}\otimes AA^*\right]$. In particular, \begin{equation}\label{nonzero}q(1,i,0)=\dfrac{\det(A^*A)}{4^nb^{2n}}h(i)=\dfrac{(-1)^{(m-2)n}\det(A^*A)^{m-1}}{4^nb^{2n}}.\end{equation}
Now, $q(1,i,z)=\dfrac{\det({\cal K}_m(A)+zI_{mn})}{\det({\cal K}_2(A)+zI_{2n})}=\dfrac{z^{mn}}{z^{2n}}=z^{(m-2)n}$, and so $q(1,i,0)=0$ since $m\geq 3$. Hence, \eqref{nonzero} implies that $A$ is singular.
\end{proof}
\begin{theorem}\label{main}
Let $m,n\in \mathbb N$ where $m\geq 2$ and $A\in M_n  $. If $A$ is nonsingular, then the following are equivalent:
\begin{enumerate}
\item[(i)] $W(\mathcal{K}_m(A)))$ is a circular disk centered at $0$.
\item[(ii)] The boundary of $W(\mathcal{K}_m(A))$ contains an elliptic arc.
\item[(iii)] $m=2$.
\end{enumerate}
\end{theorem}
\begin{proof}
The implications (i)$\Rightarrow$(ii) and (iii)$\Rightarrow$(i) are clear. It suffices to show (ii)$\Rightarrow$(iii). If $m\geq 3$, then Lemma \ref{singular} implies that $A$ is singular, which is a contradiction. Thus, $m=2$.\end{proof}

The next example illustrates that the nonsingularity assumption on $A$ cannot be removed. 

\begin{example}
\normalfont
If $A=J_2(0)$, then $\mathcal{K}_3(A)$ is unitarily similar to $J_2(0)\oplus J_2(0)\oplus 0_2$. Since $0\in W(J_2(0))$, $W(\mathcal{K}_3(A))=W(J_2(0))$ and so $W(\mathcal{K}_3(A))$ is a circular disk. On the other hand, if $B=\textup{diag}(1,0)$, then $\mathcal{K}_3(B)$ is unitarily similar to $\mathcal{K}_3(1)\oplus 0_3$. Note that $0\in W(\mathcal{K}_3(1))$, and so $W(\mathcal{K}_3(B))=W(\mathcal{K}_3(1))$ is not a circular disk due to \cite[Theorem 2.3]{hG13} or Theorem \ref{main}.
\end{example}




\begin{thebibliography}{00}

\bibitem{mC06}
M.-D.~Choi, D. W.~Kribs, and K.\.Zyczkowski.
\newblock Higher-rank numerical ranges and compression problems.
\newblock{\em Linear Algebra Appl.}, 418:828--839, 2006.

\bibitem{mC062}
M.-D.~Choi, D. W.~Kribs, and K.\.Zyczkowski.
\newblock Quantum error correcting codes from the compression formalism.
\newblock{\em Report. Math. Phys.}, 58:77-91, 2006.

\bibitem{gF18}
G.~Fikioris.
\newblock Spectral properties of Kac-Murdock-Szeg\"o matrices with a complex parameter.
\newblock {\em Linear Algebra Appl.}, 533:182--210, 2018.


\bibitem{gF20}
G.~Fikioris.
\newblock Eigenvalue bifurcations in Kac-Murdock-Szeg\"o matrices with a complex parameter.
\newblock {\em Linear Algebra Appl.}, 607:118--150, 2020.

\bibitem{gF19}
G.~Fikioris and Th.K.~Mavrogordatos.
\newblock Double, borderline, and extraordinary eigenvalues of Kac-Murdock-Szeg\"o matrices with a complex parameter.
\newblock {\em Linear Algebra Appl.}, 575:314--333, 2019.

\bibitem{gF21}
G.~Fikioris and C.~Papapanos.
\newblock Eigenvalue contour lines of Kac-Murdock-Szeg\"o matrices with a complex parameter.
\newblock {\em Linear Algebra Appl.}, 629:87--111, 2021.




\bibitem{hG10}
H.~Gaaya.
\newblock On the Numerical Radius of the Truncated Adjoint Shift.
\newblock {\em Extracta Math.}, 25:165-182, 2010.


\bibitem{hG14b}
H.-L.~Gau, K.-Z.~Wang, and P.Y.~Wu.
\newblock{Zero-dilation index of a finite matrix}.
\newblock{\em Linear Algebra Appl.}, 440:111-124, 2014.

\bibitem{hG13}
H.-L.~Gau and P.Y.~Wu.
\newblock Numerical ranges of KMS matrices.
\newblock {\em Acta Sci. Math. (Szeged)}, 79:583--610, 2013.

\bibitem{hG14a}
H.-L.~Gau and P.Y.~Wu.
\newblock Zero-dilation indices of KMS Matrices.
\newblock {\em Ann. Funct. Anal.}, 5:30--35, 2014.


\bibitem{hG16}
H.-L.~Gau  and P.Y.~Wu.
\newblock Zero-dilation index of $S_n$-matrix and Companion Matrix.
\newblock {\em Electron. J. Linear Algebra}, 31:666--678, 2016.



\bibitem{matpoly}
 I.~Gohberg, P.~Lancaster, and L.~Rodman.
 \newblock {\em Matrix Polynomials}.
 \newblock SIAM, Philadelphia, USA, 2009.
\bibitem{HJ1}
R.A.~Horn and C.R.~Johnson.
\newblock  {\em Matrix Analysis}.
\newblock  {Second Edition}.
\newblock Cambridge University Press, New York, 2013.

\bibitem{HJ2}
R.A.~Horn and C.R.~Johnson.
\newblock  {\em Topics in Matrix Analysis}.
\newblock Cambridge University Press, New York, 1991.


\bibitem{KMS}
M.~Kac, W.L.~Murdock, and G.~Szeg\"o.
\newblock On the eigenvalues of certain Hermitian forms.
\newblock {\em J. Rational Mech. Anal.}, 2:767--800, 1953.

\bibitem{kipp}
R.~Kippenhahn.
\newblock \"Uber den Wertevorrat einer Matrix.
\newblock {\em Math. Nachr.}, 6:193--228, 1951.
\newblock (English translation: P.F.~Zachlin and M.E.~Hochstenbach. On the numerical range of a matrix, {\em Linear Multilinear Algebra}, 56:185-225, 2008.)

\bibitem{fK92}
F.~Kirwan.
\newblock  {\em Complex Algebraic Curves}.
\newblock Cambridge University Press, New York, 1992.


\bibitem{cL08}
C.-K.~Li and N.-S.~Sze.
\newblock{Canonical forms, Higher Rank Numerical Ranges, Totally isotropic subspaces, and Matrix Equations}.
\newblock {\em Proc. Amer. Math. Soc.}, 136:3013--3023, 2008.



\bibitem{hW08}
H.J.~Woerdeman.
\newblock The higher rank numerical range is convex.
\newblock {\em Linear Multilinear Algebra}, 56:65--67, 2008.






















\end{thebibliography}



\bigskip
{\bf Acknowledgment.} 
The author thanks the generous support of U.P. Diliman Mathematics Foundation, Inc. MacArthur and Josefina Delos Reyes Research Grant and the donors of the U.P. Diliman Francisco and Aurora Pizana Mamaril Professorial Chair. 

\end{document}